\documentclass[12pt]{article}

\usepackage{amsmath,amsthm,amsfonts,latexsym,amscd,amssymb}

\def\qed{{\hbadness=10000\hfill\ \vbox{\hrule height.09ex
     \hbox{\vrule width.09ex height1.55ex depth.2ex \kern1.8ex
     \vrule width.09ex height1.55ex depth.2ex}\hrule height.09ex}\break
     \bigskip}}
\newtheorem{theorem}{Theorem}[section]
\newtheorem{lemma}[theorem]{Lemma}

\newtheorem{definition}[theorem]{Definition}
\newtheorem{observation}[theorem]{Observation}

\newtheorem{remark}[theorem]{Remark}
\newtheorem{example}[theorem]{Example}

\begin{document}
\title{An Example of a Right Loop Admitting Only Discrete Topolization}

\author{Vipul Kakkar\\
School of Mathematics, Harish-Chandra Research Insitute \\
Allahabad (India)\\
Email: vplkakkar@gmail.com}

\date{}
\maketitle

\begin{abstract}
Being motivated by \cite{sk} and \cite{nms}, an example of a right loop admitting only discrete topolization is given.  
\end{abstract}
\section{Introduction}
\begin{definition} \label{bd1}
A groupoid $(S, \circ)$ is called  a \textbf{right loop} (resp. \textbf{left loop}) if
\begin{enumerate}
 \item for~all $x, y \in S$,  the equation $X \circ x = y$ (resp. $x \circ X=y$) , where $X$ is unknown in the equation, has a unique solution in $S$. In notation we write it as $X=y/x$ (resp. $X=x \backslash y$) 
 \item if there exists $1 \in S$  such that $1\circ x = x\circ 1 =  x$ for~all $x \in S$.  
\end{enumerate}
\end{definition}
Let $(S,\circ)$ be a right loop (resp. left loop). For $u \in S$ the map $R_u^{\circ}:S\rightarrow S$ (resp. $L_u^{\circ}:S\rightarrow S$) defined by $R_u^{\circ}(x)=x \circ u$ (resp. $L_u^{\circ}(x)=u \circ x$) is a bijection on $S$. We will drop the superscript, if the binary operation is clear.      
\begin{definition}\label{bd2} A groupoid $(S, \circ)$ is called  a \textbf{loop} if it is right loop as well as left loop. A loop $S$ is said to be \textbf{commutative} if $x \circ y=y \circ x$ for all $x,y \in S$.  
\end{definition}
\begin{definition}\label{bd3} A loop $S$ is said to be \textbf{inverse property loop (I.P. loop)} if for each $x \in S$ there exists $x^{-1} \in S$ such that $x^{-1} \circ (x \circ y)=y$ and $(y \circ x) \circ x^{-1}=y$ for all $y \in S$. 
\end{definition}
Let $G$ be a group and $H$ be a subgroup of $G$. A \textit{normalized right transversal (NRT)} $S$ of $H$ in $G$ is a subset of $G$ obtained by choosing one and only one element from each right coset of $H$ in $G$ and $1 \in S$. Then $S$ has a induced binary operation $\circ$ given by $\{x \circ y\}=Hxy \cap S$, with respect to which $S$ is a right loop with identity $1$, (see \cite[Proposition 4.3.3, p.102]{smth},\cite{rltr}). Conversely, every right loop can be embedded as an NRT in a group with some universal property (see \cite[Theorem 3.4, p.76]{rltr}).

Let $\mathcal{T}(G,H)$ denote the set of all NRTs of $H$ in $G$. Let $S \in \mathcal{T}(G, H)$ and $\circ$ be the induced binary operation on $S$. Let $x, y\in S$ and $h\in H$. Then $x.y = f(x, y)(x\circ y)$ for some $f(x, y)\in H$ and $x\circ y\in S$. Also $x.h=\sigma_x (h) x\theta h $ for some $\sigma_x (h)\in H $ and $x\theta h \in S$. This gives us a map $f:S\times S\rightarrow H$ and a map $\sigma :S\rightarrow H^H$ defined by $f((x,y))=f(x, y)$ and $\sigma (x)(h)=\sigma_x (h)$. Also $\theta$ is a right action of $H$ on $S$. The quadruple $(S,H,\sigma,f)$ is a $c$-groupoid (see \cite[Definition 2.1, p. 71]{rltr}). Infact, every $c$-groupoid comes in this way (see \cite[Theorem 2.2, p.72]{rltr}). The same is observed in \cite{km} but with different notations (see \cite[Section 3, p. 289]{km}). 

\begin{definition}\label{bd4}
A right loop $(S,\circ)$ is said to be topological right loop if $S$ is a topological space and the operations $\circ$ and $/$ are continuous.
\end{definition}

Let $x,y \in S$. Then the map $t\mapsto (y/x) \circ t$ from $S$ to $S$ takes $x$ to $y$.
\begin{observation}\label{ob1} Let $(S,\circ)$ be a topological right loop. Then
\begin{enumerate} \item[(i)] \textit{the topological right loop $S$ is homogeneous.}  For any $x,y \in S$, the map $t\mapsto (y/x) \circ t$ is a homeomorphism sending $x$ to $y$.

\item[(ii)] \textit{every $T_1$ topological right loop is Hausdorff} Note that the map $f:S \times S\rightarrow S$ defined by $f(x,y)=x/y$ is countinuous, where $x,y \in S$. Then $\Delta=\{(x,x)|x \in S\}=f^{-1}({1})$ is a closed subset of $S \times S$. 
\end{enumerate} 
\end{observation}
\begin{remark} From now onwards, we will assume that our topological right loop is $T_1$.
\end{remark}

\section{An Example of a Right Loop Admitting Only Discrete Topolization}
Let $(U,\circ)$ be a loop. Let $e$ denote the identity of the loop $U$. Let $B \subseteq U \setminus \{e\}$ and $\eta \in Sym(U)$ such that $\eta(e)=e$. Define an operation $\circ^{\prime}$ on the set $U$ as 
\begin{equation}\label{ex1}
x \circ^{\prime} y = \left\{
\begin{array}{l l}
x \circ y & \qquad \mbox{if $y \notin B$}\\
y \circ \eta(x) & \qquad \text{if $y \in B$}\\
\end{array} \right. 
\end{equation}
It can be checked that $(U,\circ^{\prime})$ is a right loop. Let us denote this right loop as $U^B_{\eta}$. If $B=\emptyset$, then $U^B_{\eta}$ is the loop $U$ itself. If $\eta$ is fixed, then we will drop the subscript $\eta$. It can be checked that if $y \notin B$, then $R_y^{\circ^{\prime}}=R_y^{\circ}$ and if $y \in B$, then $R_y^{\circ^{\prime}}=L_y^{\circ}\eta$.

In following example, we will observe that right loop structure on each NRTs in infinite dihedral group of non-normal subgroup of order $2$ can be obtained in the manner defined in the equation (\ref{ex1}).
 
\begin{example}  Let $U=\mathbb{Z}$, the infinite cyclic group. Define a map $\eta: \mathbb{Z}\rightarrow \mathbb{Z}$ by $\eta(i)=-i$, where $i \in \mathbb{Z}$. Note that $\eta$ is a bijection on $\mathbb{Z}$. Let $\emptyset \neq B \subseteq \mathbb{Z} \setminus \{0\}$. We denote 
${\mathbb{Z}}^{B}_{\eta}$ by ${\mathbb{Z}}^{B}$.

Let $G=D_{\infty}=\langle x,y|x^2=1, xyx=y^{-1} \rangle$ and $H=\{1,x\}$. Let $N=\langle y \rangle$. Let $\epsilon :N\rightarrow H$ be a function with $\epsilon (1)=1$. Then $T_{\epsilon}=\{\epsilon(y^i)y^i|1 \leq i \leq n \} \in \mathcal{T}(G,H)$ and all NRTs $T \in \mathcal{T}(G,H)$ are of this form. 
Let $B=\{i \in \mathbb{Z}|\epsilon(y^i)=x\}$. Since $\epsilon$ is completely determined by the 
subset $B$, we shall denote  $T_{\epsilon}$ by $T_B$.
Clearly, 
the map $\epsilon(y^i)y^i\mapsto i$ from $T_{\epsilon}$ to $\mathbb{Z}^B$
is an isomorphism of right loops. So we may identify the right loop $T_B$ with 
the right loop ${\mathbb{Z}}^{B}$ by means of the above isomorphism. 
We observe that $T_{\emptyset}=N \cong \mathbb{Z}$.
\end{example}
\begin{lemma}\label{l1} Let $G$ be a group and $H$ a subgroup of $G$. Let $S \in \mathcal{T}(G,H)$ such that $S$ is an I.P. loop with respect to the induced binary operation on it. Then $L_{1/a}^{-1}=L_a$.
\end{lemma} 
\begin{proof} By \cite[Proposition 3.3, p. 290; Proposition 3.6, p. 293]{km}, $1/a=a^{\prime}$, where $u^{\prime}$ denote the left inverse of $u \in S$. Let $x \in S$ and $y=L_{1/a}^{-1}(x)=L_{a^{\prime}}^{-1}(x)$ $\Rightarrow x = a^{\prime} \circ y$ $\Rightarrow (a^{\prime})^{\prime} \circ x= (a^{\prime})^{\prime} \circ (a^{\prime} \circ y)$. 

By \cite[Proposition 3.6, p.293]{km}, $(a^{\prime})^{\prime}=a$. Also, since $S$ is I.P. loop, $(a^{\prime})^{\prime} \circ (a^{\prime} \circ y)=y$. Then above equation becomes $a \circ x=y\Rightarrow L_a(x)=y$. This shows that $L_{1/a}^{-1}=L_a$.
\end{proof}
Following, we will give an example of a right loop where only discrete topolization is possible:
\begin{example} Let $G_1$ be a countable group with relation $x^2=e$, where $e$ is the identity of $G_1$ and $G_2$ be a finite group with identity $1$ on which two distinct involutive automorphism $\phi$ and $\psi$ exists. Fix two distinct elements $a,b \in G_1$. Let $(L=G_1 \times G_2, \star)$ be commutative I.P. loop as in \cite{nms}.

Let $B=\{(e,1)/(a,1),(b,1)\}$. Let $\eta \in Sym(L)$ such that $\eta(e,1)=(e,1)$. Define a binary operation $\circ^{\prime}$ on $L$ as defined in the equation (\ref{ex1}). Consider a topology on $L$ so that $L$ is a topological right loop. Then the map $\alpha=R_{(b,1)}^{\circ^{\prime}}(R_{(e,1)/(a,1)}^{\circ^{\prime}})^{-1}R_{(b,1)}^{\circ^{\prime}}(R_{(e,1)/(a,1)}^{\circ^{\prime}})^{-1}$ is a homeomorphism. One can note that $\alpha=L_{(b,1)}^{\star}(L_{(e,1)/(a,1)}^{\star})^{-1}L_{(b,1)}^{\star}(L_{(e,1)/(a,1)}^{\star})^{-1}=L_{(b,1)}^{\star}L_{(a,1)}^{\star}L_{(b,1)}^{\star}L_{(a,1)}^{\star}$ (by Lemma \ref{l1}). 

Also note that $\alpha$ is same as given in \cite{nms}. This $\alpha$ moves only finitely many points. Therefore, the right loop $(L,\circ^{\prime})$ can not be topolized in non-discrete manner.  
\end{example} 

Let $S$ be a right loop. Then the group generated by $R_a$ for all $a \in S$ is called as the \textit{right multiplication group of $S$}. At the end, we ask following question:

\textbf{Question:} Let $S$ be a right loop admitting only discrete topology. What topology can the right multiplication group of $S$ can have?


\end{document}